\documentclass[12pt]{amsart}
\usepackage[colorlinks=true,citecolor=magenta,linkcolor=magenta]{hyperref}
\usepackage{amssymb,verbatim,amscd,amsmath,graphicx}
\usepackage{enumerate}
\usepackage{graphicx}
\usepackage{caption}
\usepackage{subcaption}
\usepackage{pdfsync}
\usepackage{color,colortbl}
\usepackage{fullpage}
\usepackage{bbm}
\usepackage{comment}

\numberwithin{equation}{section}
\newtheorem{theorem}{Theorem}[section]
\newtheorem{lemma}[theorem]{Lemma}

\newtheorem{corollary}[theorem]{Corollary}

\theoremstyle{definition}
\newtheorem{definition}[theorem]{Definition}

\newtheorem{def-prop}[theorem]{Definition-Proposition}

\newtheorem{remark}[theorem]{Remark}
\newtheorem{example}[theorem]{Example}

\newtheorem*{acknowledgement}{Acknowledgements}

\newtheorem{question}[theorem]{Question}

\DeclareMathOperator{\Tor}{Tor}

\DeclareMathOperator{\reg}{reg}
\DeclareMathOperator{\depth}{depth}

\DeclareMathOperator{\NP}{NP}

\newcommand{\ZZ}{{\mathbb Z}}
\newcommand{\NN}{{\mathbb N}}
\newcommand{\QQ}{{\mathbb Q}}
\newcommand{\RR}{{\mathbb R}}

\newcommand{\kk}{{\mathbbm k}}

\def\mm{{\frak m}}

\def\a{{\bf a}}

\def\e{{\bf e}}

\def\y{{\bf y}}

\def\1{{\bf 1}}
\def\0{{\bf 0}}

\begin{document}
	
	\title{Integral closures of powers of sums of ideals}

	\author[Banerjee]{Arindam Banerjee}
	\address{Department of Mathematics, Indian Institute of Technology, Kharagpur, West Bengal, India}
	\email{123.arindam@gmail.com}
	
	\author[H\`a]{T\`ai Huy H\`a}
	\address{Mathematics Department, Tulane University, 6823 St. Charles Avenue, New Orleans, LA 70118, USA}
	\email{tha@tulane.edu}
\thanks{Corresponding author: T\`ai Huy H\`a, tha@tulane.edu.}
	
	\subjclass[2020]{13C13, 90C05, 13D07}
    \keywords{Integral closure, monomial ideal, sum of ideals, power of ideal, rational power, depth, regularity}
	
	\begin{abstract}
	Let $\kk$ be a field, let $A$ and $B$ be polynomial rings over $\kk$, and let $S= A \otimes_\kk B$. Let $I \subseteq A$ and $J \subseteq B$ be monomial ideals. We establish a binomial expansion for rational powers of $I+J \subseteq S$ in terms of those of $I$ and $J$. Particularly, for $u \in \QQ_+$, we prove that
	$$(I+J)_u = \sum_{0 \le \omega \le u, \ \omega \in \QQ} I_\omega J_{u-\omega},$$
	and that the sum on the right hand side is a finite sum. This finite sum can be made more precise using jumping numbers of rational powers of $I$ and $J$. We further give sufficient conditions for this formula to hold for the integral closures of powers of $I+J$ in terms of those of $I$ and $J$.
	Under these conditions, we provide explicit formulas for the depth and regularity of $\overline{(I+J)^k}$ in terms of those of powers of $I$ and $J$. %Our methods are linear algebra in nature, making use of the characterization of rational powers of a monomial ideal via optimal solutions to linear programming problems.
	\end{abstract}
	
	\maketitle
	
%%%%%%%%%%%%%%%%%%%%%%%%%%%%%%%%%

\section{Introduction} \label{sec.intro}

Let $\kk$ be a field. Let $A$ and $B$ be polynomial rings over $\kk$, and let $I \subseteq A$ and $J \subseteq B$ be ideals. A binomial expansion for symbolic powers of the sum $I+J$ in $S = A \otimes_\kk B$ was given by the second author and various co-authors in \cite{HJKN2022, HNTT2020}. Particularly, it was shown that, for any positive integer $k \in \NN$,
\begin{align}
	(I+J)^{(k)} = \sum_{\ell=0}^k I^{(\ell)} \cdot J^{(k-\ell)}. \label{eq.bin}
\end{align}
This formula was proved for symbolic powers defined using minimal primes in \cite{HNTT2020}. It was established for symbolic powers defined using all associated primes and, more generally, for saturated powers recently in \cite{HJKN2022}. The formula has been quite well received and seen many applications since its discovery (cf. \cite{BHJT, EH2020, ERT2020, KKS2021, LS2021, MV2021, OR2019, SF2020, SW2019, W2018}).

It is not known if the integral closures of powers of an ideal could be realized as saturated powers. Thus, a natural question arises: \emph{does a similar binomial expansion exist for the integral closures of powers of sums of ideals?} We shall address this question in this paper.

Simple examples exist to illustrate that the binomial expansion (\ref{eq.bin}) does not hold in general when symbolic powers are replaced by the integral closures of powers. For instance, by taking $I = (x^2) \subseteq \kk[x] = A$ and $J = (y^2) \subseteq B = \kk[y]$, it is easy to see that $xy \in \overline{I+J} \subseteq \kk[x,y]$, while $\overline{I} + \overline{J} = (x^2,y^2)$ does not contain $xy$. Particularly,
$$\overline{I+J} \not= \overline{I} + \overline{J}.$$
On the other hand, a recent result of Mau and Trung \cite[Theorem 2.1]{MT2021} showed that if $I \subseteq A$ is a \emph{normally torsion-free} squarefree monomial ideal and $J \subseteq B$ is an arbitrary monomial ideal then, for any $k \in \NN$,
\begin{align}
	\overline{(I+J)^k} = \sum_{\ell=0}^k \overline{I^\ell} \cdot \overline{J^{k-\ell}}. \label{eq.binIC}
\end{align}
(Note that if both $I$ and $J$ are normally torsion-free squarefree monomial ideals, then $I+J$ is normally torsion-free by
\cite[Corollary 5.6]{SVV}; see also \cite[Corollary 5.3 and Theorem 5.4]{SVV} for
more information about the addition of normally torsion-free ideals. Thus, in this case, 
equality (\ref{eq.binIC}) becomes the usual expansion of ordinary powers.)
It is therefore desirable to establish new binomial expansions that work for all monomial ideals, and to identify classes of monomial ideals for which the binomial expansion (\ref{eq.binIC}) holds.

To search for a binomial expansion that holds for all monomial ideals, our solution is to focus instead on \emph{rational powers}. Let $u = \frac{p}{q} \in \QQ_+$ be any positive rational number, with $p,q \in \NN$ and $q \not= 0$. Following \cite[Definition 10.5.1]{SH2006}, the \emph{$u$-th rational power} of an ideal $I$ in a domain $A$ is defined to be
$$I_u = \{x \in A ~\big|~ x^q \in \overline{I^p}\}.$$
This definition of $I_u$ does not depend on the particular presentation $u = \frac{p}{q}$. Obviously, if $u$ is a positive integer then $I_u = \overline{I^u}$ is the integral closure of $I^u$. Rational powers have been extended to \emph{real powers} in a recent work of Dongre \emph{et. al.} \cite{D+2021}.

Our first main theorem reads as follows.

\medskip

\noindent\textbf{Theorem \ref{thm.binRAT}.} Let $I \subseteq A$ and $J \subseteq B$ be monomial ideals. Let $u \in \QQ$ be any positive rational number. Then,
$$(I+J)_u = \sum_{0 \le \omega \le u, \ \omega \in \QQ} I_\omega \cdot J_{u-\omega},$$
and the sum on the right hand side is a finite sum.

\medskip

Theorem \ref{thm.binRAT} particularly exhibits a reason why we should not expect the binomial expansion (\ref{eq.binIC}) to hold for all monomial ideals in general --- there are missing terms in the right hand side of (\ref{eq.binIC}) when $u=k$ is an integer.

Our methods are based on \cite{DTWWV, HT2019}, where the membership in the integral closure $\overline{I^k}$ was characterized in terms of the optimal solution to linear programming problems associated to $I$. Specifically, let $M$ be the matrix whose columns are exponent vectors of the (unique set of) minimal monomial generators of $I$ and assume that $M$ is of size $n \times m$. For $\a \in \ZZ_{\ge 0}^n$, consider the following linear programming problem:
\begin{align*}
	{\rm (\star)} & \left\{\begin{array}{l} \text{maximize } \1^{m} \cdot \y, \\ \text{subject to } M \cdot \y \le \a, \y \in \RR^{m}_{\ge 0}. \end{array}\right.
%\quad \text{ and } \\
%    {\rm (b)} & \left\{\begin{array}{l} \text{minimize } \a \cdot \z, \\ \text{subject to } M^{\text{T}} \cdot \z \ge \1^m, \z \in \RR^{n}_{\ge 0}. \end{array}\right.
\end{align*}
Let $\nu^*_\a(I)$ be the optimal solution to ($\star$). It was shown in \cite[Proposition 1.1]{HT2019} (see also \cite[Proposition 3.5 and Remark 3.6]{DTWWV}) that $x^\a \in \overline{I^k}$ if and only if $\nu^*_\a(I) \ge k$. We provide a similar criterion for the membership of a rational power $I_u$; see Lemma \ref{lem.membershipRAT}.

The finite sum on the right hand side of Theorem \ref{thm.binRAT} can be made more precise in terms of jumping numbers of $I$ and $J$. The jumping numbers of a monomial ideal were defined in \cite{D+2021} as a means to identify different real powers of the ideal. It turns out that these jumping numbers are rational. For a fixed $u \in \QQ_+$, a rational number $\theta \in [0,u]$ is called a \emph{jumping number on the interval $[0,u]$} of a monomial ideal $I$ if either $\theta = u$ or $I_\theta \not= I_{\theta'}$ for all $\theta' > \theta$ (see also \cite[Corollary 5.7]{D+2021}).

Thanks to an anonymous referee's observation and suggestion, we prove the following result.

\medskip

\noindent\textbf{Theorem \ref{thm.jump}.} Let $I \subseteq A$ and $J \subseteq B$ be monomial ideals. Let $u \in \QQ$ be any positive rational number. Then,
	$$(I+J)_u = \sum_{\substack{\omega \text{ is a jumping number} \\ \text{of $I$ on } [0,u]}} I_\omega \cdot J_{u - \omega} = \sum_{\substack{\theta \text{ is a jumping number} \\ \text{of $J$ on } [0,u]}} I_{u-\theta} \cdot J_\theta.$$

\medskip

Theorem \ref{thm.jump} is achieved by observing that distinct terms on the right hand side of the binomial expansion established in Theorem \ref{thm.binRAT} are exactly those with different rational powers of $I$ or of $J$.

As shown in Theorems \ref{thm.binRAT} and \ref{thm.jump}, we cannot expect the binomial expansion (\ref{eq.binIC}) to hold for all monomial ideals. Corollary \ref{thm.BinomialExpansion} is a special case of Theorem \ref{thm.binRAT} and gives an improvement of the aforementioned result of Mau and Trung \cite[Theorem 2.1]{MT2021}. 
Making use of the jumping numbers of powers of $I$ and $J$, Corollary \ref{cor.bin_jump} is a consequence of Theorem \ref{thm.jump} and presents another sufficient condition for the binomial expansion (\ref{eq.binIC}) to hold. %Also thanks to an anonymous referee's suggestion, we establish the following result.

\medskip 

\noindent\textbf{Corollaries \ref{thm.BinomialExpansion} and \ref{cor.bin_jump}.} Let $I \subseteq A$ and $J \subseteq B$ be monomial ideals, and let $k \in \NN$. Suppose that at least one of the following conditions holds:
\begin{enumerate}
	\item for every nonnegative integral vector $\alpha$, $\nu^*_\alpha(I) \in \ZZ$; or
	\item the jumping numbers on $[0,k]$ of either $I$ or $J$ are all integers.
\end{enumerate}
Then, we have
 $$\overline{(I+J)^k} = \sum_{\ell = 0}^k \overline{I^\ell} \cdot \overline{J^{k-\ell}}.$$
 
\medskip

Having a binomial expansion as in (\ref{eq.binIC}) for $\overline{(I+J)^k}$ allows us to estimate important algebraic invariant, such as the depth and the regularity, of $S/\overline{(I+J)^k}$. Particularly, we exhibit explicit formulas for the depth and regularity of $S/\overline{(I+J)^k}$ in terms of those of the integral closures of powers of $I$ and $J$, under the sufficient conditions in Corollaries \ref{thm.BinomialExpansion} and \ref{cor.bin_jump}. Such formulas for the ordinary and symbolic powers of $(I+J)$ were given in \cite{HNTT2020, HTT2016, NV2019}. A formula for the integral closure of powers of $(I+J)$ would be desirable, for instance, as stated in \cite{MV2021}.

\medskip

\noindent\textbf{Theorem \ref{thm.DepthReg}.} Let $I \subseteq A$ and $J \subseteq B$ be monomial ideals, and let $k \in \NN$. Suppose that at least one of the following conditions holds:
\begin{enumerate}
	\item for every nonnegative integral vector $\alpha$, $\nu^*_\alpha(I) \in \ZZ$; or
	\item the jumping numbers on $[0,k]$ of either $I$ or $J$ are all integers.
\end{enumerate}
Then, we have
\begin{enumerate}
	\item[(1)] $\depth S/\overline{(I+J)^k} = $ \newline
	\hspace*{3ex} ${\displaystyle \min_{\substack{i \in [1,k-1] \\ j \in [1,k]}} \{\depth A/\overline{I^{k-i}} + \depth B/\overline{J^i} + 1, \depth A/\overline{I^{k-j+1}} + \depth B/\overline{J^j}\}}$,
	\item[(2)] $\reg S/\overline{(I+J)^k} = $ \newline
	\hspace*{3ex} ${\displaystyle \max_{\substack{i \in [1,k-1] \\ j \in [1,k]}} \{\reg A/\overline{I^{k-i}} + \reg B/\overline{J^i} + 1, \reg A/\overline{I^{k-j+1}} + \reg B/\overline{J^j}\}}$.
\end{enumerate}

\medskip

Our approach to proving Theorem \ref{thm.DepthReg} is similar to that of \cite[Theorems 4.2 and 5.3]{HNTT2020}. Particularly, by setting
$$P_{k,t} = \overline{I^k} \cdot \overline{J^0} + \overline{I^{k-1}} \cdot \overline{J} + \dots + \overline{I^{k-t}} \cdot \overline{J^t},$$
for $0 \le t \le k$, and observing that
\begin{enumerate}
\item[(a)] $P_{k,t} = P_{k,t-1} + \overline{I^{k-t}} \cdot \overline{J^t}$, and
\item[(b)] $P_{k,t-1} \cap \overline{I^{k-t}} \cdot \overline{J^t} = \overline{I^{k-t+1}} \cdot \overline{J^t},$
\end{enumerate}
one direction of the inequality in Theorem \ref{thm.DepthReg} follows from standard short exact sequences:
\begin{align}
0 \longrightarrow R\big/P_{k,t-1} \cap \overline{I^{k-t}} \rightarrow R/P_{k,t-1} \oplus R\big/\overline{I^{k-t}} \cdot \overline{J^t} \longrightarrow R/P_{k,t} \longrightarrow 0. \label{eq.sesSplitting}
\end{align}
To establish the reverse inequality, we show that the decomposition $P_{k,t} = P_{k,t-1} + \overline{I^{k-t}} \cdot \overline{J^t}$ is a \emph{Betti splitting}, a notion defined by Francisco, H\`a and Van Tuyl \cite{FHVT2009} to guarantee that the inequality between Betti numbers that result from the exact sequences in (\ref{eq.sesSplitting}) are in fact equality. To accomplish this last step, we exhibit that the filtration $\{\overline{I^k}\}_{k \in \NN}$ and $\{\overline{J^k}\}_{k \in \NN}$ are \emph{Tor-vanishing}, in the sense of Nguyen and Vu \cite{NV2019}.

 \begin{acknowledgement} The authors thank Rafael H. Villarreal for pointing us to their paper \cite{DTWWV}. The authors thank an anonymous referee for suggesting to connect our results to jumping numbers defined in \cite{D+2021} and, thus, improving some of our statements. The authors thank an anonymous referee and Jonathan Monta\~no for providing us with Example \ref{ex.notNess}. The first author acknowledges supports from DST INSPIRE Faculty Fellowship and CPDA of IIT Kharagpur. The second author is partially supported by Louisiana Board of Regents and the Simons Foundation.
\end{acknowledgement}

%%%%%%%%%%%%%%%%%%%%%%%%%%%%%%%

%\section{Preliminaries} \label{sec.prel}

%Let $A = \kk[X_1, \dots, X_r]$, $B = \kk[Y_1, \dots, Y_s]$, and $S = A \otimes_\kk B$. Let $I \subseteq A$ and $J \subseteq B$ be monomial ideals. Let $I+J$ denote the sum of extensions of $I$ and $J$ in $S$. Observe that for an ideal $I \subseteq A$, the extension of $\overline{I}$ in $S$ is the same as the integral closure of $IS$ in $S$, i.e., $\overline{I}S = \overline{IS}$ (\textcolor{red}{Why?}), so by abusing of notation, we shall write $\overline{I}$ to indicate the integral closure of $I$ considered both as an ideal in $A$ and as its extension in $S$.

%%%%%%%%%%%%%%%%%%%%%%%%%%%%%%%%%

\section{Binomial expansion of integral closures of powers} \label{sec.bin}

Throughout the paper, let $A = \kk[X_1, \dots, X_r]$, $B = \kk[Y_1, \dots, Y_s]$, and $S = A \otimes_\kk B$. Let $I \subseteq A$ and $J \subseteq B$ be monomial ideals. By abusing notation, we shall write $I$ and $J$ also for their extensions in $S$.

Observe that the extension of $\overline{I}$ in $S$ is the same as the integral closure of $IS$ in $S$, i.e., $\overline{I}S = \overline{IS}$. This follows, for instance, from  \cite[Corollary 19.5.2]{SH2006}, since $S$ is a normal extension of $A$. Therefore, also by abusing notation, we shall write $\overline{I}$ and $\overline{J}$ to refer to the integral closures of $I$ and $J$, considered both as ideals in $A$ and $B$, respectively, and as their extensions in $S$.

Recall that for an ideal $I \subseteq A$ and a positive rational number $u = \frac{p}{q}$, with $p, q \in \NN$ and $q \not= 0$, the \emph{$u$-th rational power} of $I$ is
$$I_u= \{x \in A ~\big|~ x^q \in \overline{I^p}\}.$$
For monomial ideals, rational powers were extended to real powers in \cite{D+2021}. Particularly, for a monomial ideal $I \subseteq A$ and $u \in \RR_{\geq 0}$, the $u$-th \emph{real power} of $I$ is defined to be
$$I_u=\{x^\a \in A ~\big|~ \a\in u.\NP(I)\cap \ZZ^r_{\ge 0}\},$$
where $\NP(I)$ is the Newton polyhedron of $I$ and, for $\a = (a_1, \dots, a_r) \in \ZZ^r_{\ge 0}$, $x^\a$ represents the monomial $X_1^{a_1} \cdots X_r^{a_r}$ in $A$. 

The following membership criterion for rational powers is similar to that of \cite[Proposition 1.1]{HT2019} and \cite[Proposition 3.5 and Remark 3.6]{DTWWV}.

\begin{lemma}
	\label{lem.membershipRAT}
	Let $I \subseteq A$ be any monomial ideal and let $\a \in \ZZ_{\ge 0}^r$. Let $\frac{p}{q}$ be any rational number, where $p,q \in \NN$ and $q \not= 0$. Then, $x^\a \in I_{\frac{p}{q}}$ if and only if $\nu^*_{\a}(I) \ge \frac{p}{q}$.
\end{lemma}

\begin{proof} By definition, $x^\a \in I_{\frac{p}{q}}$ if and only if $\left(x^\a\right)^q \in \overline{I^p}$. By \cite[Proposition 1.1]{HT2019}, this is the case if and only if $\nu^*_{q\cdot \a}(I) \ge p$ or, equivalently, $\nu^*_\a(I) \ge \frac{p}{q}$.
\end{proof}

Observe that if $I$ is a monomial ideal then $\overline{I^p}$ is a monomial ideal. This, particularly, implies that the $u$-th rational power $I_u$, for $u = \frac{p}{q}$, is also a monomial ideal. Our first main result is stated as follows.

\begin{theorem}
	\label{thm.binRAT}
	Let $I \subseteq A$ and $J \subseteq B$ be monomial ideal. Let $u \in \QQ$ be any positive rational number. Then,
	$$(I+J)_u = \sum_{\substack{0\le \omega \le u, \ \omega \in \QQ}} I_\omega \cdot J_{u-\omega},$$
	and the sum on the right hand side is a finite sum.
\end{theorem}

\begin{proof} Fix a rational number $0 \le \omega \le u$. By using the same denominator, without loss of generality, we may assume that $u = \frac{p}{q}$, and $\omega =  \frac{k}{q}$ for some $0 \le k \le p$. Consider arbitrary monomials $x^\a \in I_\omega$ and $x^\e \in J_{u-\omega}.$ By definition, we have
	$$\left(x^\a\right)^q \in \overline{I^k} \text{ and } \left(x^\e\right)^q \in \overline{J^{p-k}}.$$
	It follows that $\left(x^{\a+\e}\right)^q \in \overline{I^k} \cdot \overline{J^{p-k}} \subseteq \overline{(I+J)^p}$, where the later inclusion is a consequence of \cite[Proposition 10.5.2]{SH2006}. Particularly, $x^{\a+\e} \in (I+J)_{\frac{p}{q}}$. This proves that $I_\omega J_{u-\omega} \subseteq (I+J)_u$. Since this holds for all rational numbers $0 \le \omega \le u$, we obtain the inclusion
	$$(I+J)_u \supseteq \sum_{\substack{0\le \omega \le u, \ \omega \in \QQ}} I_\omega \cdot J_{u-\omega}.$$
	
	We shall proceed to get the reverse inclusion. As observed before, $(I+J)_u$ is a monomial ideal. Thus, it suffices to show that all monomials in $(I+J)_u$ belong to $\sum\limits_{0 \le \omega \le u, \ \omega \in \QQ} I_\omega \cdot J_{u-\omega}$.
	
	Let $m_1$ and $m_2$ be the number of minimal generators of $I$ and $J$, respectively. Let $M_1$ be the $r \times m_1$ exponent matrix of $I$ and let $M_2$ be the $s \times m_2$ exponent matrix of $J$. Set $n = r+s$ and $m = m_1 + m_2$. Relabel the variables of $S$ to be $x_1, \dots, x_n$ (corresponding to $X_1, \dots, X_r, Y_1, \dots, Y_s$). Define $M$ to be the following $n \times m$ matrix:
	$$M = \left( \begin{array}{c|c} M_1 & \ \0 \\ \hline  \0 & M_2\end{array}\right).$$
	
	Consider any monomial $x^\a \in (I+J)_u$. We have $\left(x^\a\right)^q \in \overline{(I+J)^p}$; that is, $x^{q\cdot \a} \in \overline{(I+J)^p}.$
	By \cite[Proposition 1.1]{HT2019}, we get $\nu^*_{q \cdot \a}(M) \ge p$ or, equivalently, $\nu^*_\a(M) \ge u$. This condition states that the optimal solution to the following linear programming problem is at least $u$:
	\begin{align}
		\label{eq.LP}
		\left\{\begin{array}{l} \text{maximize } \1^m \cdot \y, \\
		 \text{subject to } M \cdot \y \le \a, \ \y \in \RR^m_{\ge 0}. 
		\end{array} \right.
	\end{align}
	
	Suppose that $\y = (y_1, \dots, y_m) \in \RR^m_{\ge 0}$ is a vector that gives the optimal solution to (\ref{eq.LP}). Then, we have
	$$M \cdot \y \le \a \text{ and } \1^m \cdot \y \ge u.$$
	Write $\a = (\alpha, \0) + (\0, \beta)$, where $\alpha \in \ZZ_{\ge 0}^r$ and $\beta \in \ZZ_{\ge 0}^s$, and $\y = (\y_1, \0) + (\0, \y_2)$, where $\y_1 \in \RR^{m_1}_{\ge 0}$ and $\y_2 \in \RR^{m_2}_{\ge 0}$. Observe that, since $M$ is a block matrix, the optimization problem (\ref{eq.LP}) is equivalent to the following two problems:
	$$ (\dagger) \left\{\begin{array}{l} \text{maximize } \1^{m_1} \cdot \y_1, \\ \text{subject to } M_1 \cdot \y_1 \le \alpha, \y_1 \in \RR^{m_1}_{\ge 0} \end{array}\right.
	\quad \text{ and } \quad (\sharp) \left\{\begin{array}{l} \text{maximize } \1^{m_2} \cdot \y_2, \\ \text{subject to } M_2\cdot \y_2 \le \beta, \y_2 \in \RR^{m_2}_{\ge 0}. \end{array}\right. $$
	
	Let $\omega$ be the optimal solution to ($\dagger$). It can be seen that the system $M_1 \cdot \y_1 \le \alpha$ (here, $\a = (\alpha, \0) + (\0, \beta)$ for $\alpha \in \ZZ_{\ge 0}^r$ and $\beta \in \ZZ_{\ge 0}^s$) consists of linear inequalities with rational coefficients, so its feasible set has rational vertices. This implies that $\omega = \nu^*_\alpha(M_1) \in \QQ$. We also have $\nu^*_\beta(M_2) = \nu^*_\a(M) - \nu^*_\alpha(M_1) \ge u-\omega \in \QQ$. It then follows from Lemma \ref{lem.membershipRAT} that $x^\alpha \in I_\omega$ and $x^\beta \in J_{u-\omega}$. If $\omega \ge u$ then we get $x^\a \in I_\omega S \subseteq I_u S$. Of $\omega \le u$ then we have $x^\a = x^\alpha \cdot x^\beta \in I_\omega J_{u-\omega}$. This is true for any monomial $x^\a \in (I+J)_u$. Hence, we obtain the inclusion
	$$(I+J)_u \subseteq \sum_{\substack{0\le \omega \le u, \ \omega \in \QQ}} I_\omega \cdot J_{u-\omega},$$
	and the first assertion is established.
	
	We continue to prove the second assertion. Observe that, by \cite[Proposition 10.5.5]{SH2006}, there exist  integers $e$ and $f$ such that every rational power $I_\omega$ and $J_\omega$ is of the form $I_{\frac{p}{e}}$ and $J_{\frac{q}{f}}$ for some $p,q \in \NN$. Particularly, it was shown in \cite[Proposition 10.5.5]{SH2006} that
	$$I_u = I_{\frac{\lceil ue\rceil}{e}} =  I_{\frac{n}{e}} \text{ and } J_u = J_{\frac{\lceil uf\rceil}{f}} = J_{\frac{m}{f}},$$
	where $n = \lceil ue\rceil$, and $m = \lceil uf\rceil$.
	
	It can be seen that $\frac{n-1}{e} < u$ and $\frac{m-1}{f} < u$. Moreover, by \cite[Proposition 10.5.2]{SH2006}, $I_\alpha \supseteq I_\beta$ if $\alpha \le \beta$. Therefore, it follows that $\{I_\omega ~\big|~ 0 \le \omega \le u\}$ coincides with $\{I_{\frac{p}{e}} ~\big|~ 0 \le p \le n\}$ and $\{J_{u-\omega} ~\big|~ 0 \le \omega \le u\}$ coincides with $\{J_{\frac{q}{e}} ~\big|~ 0 \le q \le m\}$. Hence, in the sum $\sum_{0 \le \omega \le u, \ \omega \in \QQ}I_\omega J_{u-\omega}$, only finitely many terms appear. This completes the second assertion of the theorem.
\end{proof}

%The following example illustrates Theorem \ref{thm.binRAT}.

\begin{example}
	\label{ex.notBE}
	Consider $I = (x^2) \subseteq \kk[x]$ and $J = (y^2, yz) \subseteq \kk[y,z]$. By \cite[Theorem 10.3.5 and Proposition 10.5.5]{SH2006}, the constants $e$ and $f$ as in the proof of Theorem \ref{thm.binRAT} for rational powers of $I$ and $J$ can be taken to be $e = f = 2$. Thus, Theorem \ref{thm.binRAT} gives the following binomial expansion
	$$\overline{(I+J)^2} = I_2 + I_{\frac{3}{2}} \cdot J_{\frac{1}{2}} + I_1 \cdot J_1 + I_{\frac{1}{2}} \cdot J_{\frac{3}{2}} + J_2 = \overline{I^2} + I_{\frac{3}{2}} \cdot J_{\frac{1}{2}} + \overline{I} \cdot \overline{J} + I_{\frac{1}{2}} \cdot J_{\frac{3}{2}} + \overline{J^2}.$$
	
	It can be verified directly that
	$$\overline{(I+J)^2} = (y^2z^2, y^3z, xy^2z, x^2yz, y^4, xy^3, x^2y^2, x^3y, x^4),$$
	while
	$$\overline{I^2} + \overline{I} \cdot \overline{J} + \overline{J^2} = (x^4, y^2z^2, y^3z, y^4, x^2yz, x^2y^2).$$
	Clearly, $\overline{(I+J)^2} \not= \overline{I^2} + \overline{IJ} + \overline{J^2}.$
	
	On the other hand, it is easy to see that $x^2 \in I \subseteq \overline{I}$ and $y^6 \in I^3 \subseteq \overline{I^3}$, so $x \in I_{\frac{1}{2}}$ and $y^3 \in J_{\frac{3}{2}}$. Thus, $xy^3 \in I_{\frac{1}{2}} \cdot J_{\frac{3}{2}}.$ Similarly, it can be seen that $xy^2z \in I_{\frac{1}{2}} \cdot J_{\frac{3}{2}}$ and $x^3y \in I_{\frac{3}{2}} \cdot J_{\frac{1}{2}}$.
\end{example}

It would be desirable to see if the binomial expansion for rational powers in Theorem \ref{thm.binRAT} holds for arbitrary ideals.

\begin{question}
	\label{quest.binRAT}
	Let $I \subseteq A$ and $J \subseteq B$ be arbitrary (or homogeneous) proper ideals. Is it true that for any $u \in \QQ_+$, we have
	$$(I+J)_u = \sum_{0 \le \omega \le u, \ \omega \in \QQ} I_\omega \cdot J_{u-\omega}?$$
\end{question}

With another closer look at the binomial terms in Theorem \ref{thm.binRAT}, it can be realized that the distinct powers of $I$ and $J$ correspond to jumping numbers. Thus, the finite sum in Theorem \ref{thm.binRAT} can be made more precise using jumping numbers of powers of $I$ and $J$ as follows. The authors thank an anonymous referee for pointing out this observation. (By definition, if $u$ is sandwiched between two consecutive jumping numbers $j' < j$ of a monomial ideal $I$, i.e., $j' < u < j$, then we will also consider $u$ as a \emph{jumping number of $I$ on the interval $[0,u]$}.)

\begin{theorem}
	\label{thm.jump}
	Let $I \subseteq A$ and $J \subseteq B$ be monomial ideals. Let $u \in \QQ$ be any positive rational number. Then,
	$$(I+J)_u = \sum_{\substack{\omega \text{ is a jumping number} \\ \text{of $I$ on } [0,u]}} I_\omega \cdot J_{u - \omega} = \sum_{\substack{\theta \text{ is a jumping number} \\ \text{of $J$ on } [0,u]}} I_{u-\theta} \cdot J_\theta.$$
\end{theorem}

\begin{proof} We shall establish the first equality, as the second one can be similarly handled. Note that, by \cite[Theorem 5.9(1)]{D+2021}, all jumping numbers of $I$ and $J$ are rational. Notice also that, as indicated in the proof of Theorem \ref{thm.binRAT}, the binomial expansion of $(I+J)_u$ contains all distinct rational powers in the interval $[0,u]$ of $I$. Thus, it follows that
\begin{align}
	(I+J)_u \supseteq \sum_{\substack{\omega \text{ is a jumping number} \\ \text{of $I$ on } [0,u]}} I_\omega \cdot J_{u - \omega}.\label{eq.jump1}
\end{align}
	
To prove the other containment, observe that if $j' < j$ are two consecutive jumping numbers of $I$ on $[0,u]$, then for any $j' < \omega < j$, by \cite[Lemma 5.1 and Corollary 5.7]{D+2021}, we have
$$I_\omega = I_j \text{ and } J_{u-\omega} \subseteq J_{u-j}.$$
Thus, $I_\omega \cdot J_{u-\omega} \subseteq I_j \cdot J_{u-j}$, which is included in the right hand side of (\ref{eq.jump1}). This implies that all terms in the binomial expansion of Theorem \ref{thm.binRAT} are included in the right hand side of (\ref{eq.jump1}). Therefore,
$$(I+J)_u \subseteq \sum_{\substack{\omega \text{ is a jumping number} \\ \text{of $I$ on } [0,u]}} I_\omega \cdot J_{u - \omega},$$
and the proof of the desired equality completes.
\end{proof}

As illustrated in Example \ref{ex.notBE}, when $u = k$ is a positive integer, the right hand side of the binomial expansion for $\overline{(I+J)^k} = (I+J)_u$ given in Theorem \ref{thm.binRAT} has more terms than just the integral closures of powers of $I$ and $J$, which were as expressed in (\ref{eq.binIC}). This explains why we cannot expect (\ref{eq.binIC}) to hold for arbitrary monomial ideals. Our next result gives a sufficient condition for the binomial expansion (\ref{eq.binIC}) to hold and generalizes \cite[Theorem 2.1]{MT2021} to a larger class of ideals.

\begin{corollary}
	\label{thm.BinomialExpansion}
	Let $I \subseteq A$ and $J \subseteq B$ be monomial ideals. Suppose that for every $\alpha \in \ZZ_{\ge 0}^r$, $\nu^*_\alpha(I) \in \ZZ$. Then, for any $k \in \NN$, we have
	$$\overline{(I+J)^k} = \sum_{\ell=0}^k \overline{I^\ell}\cdot \overline{J^{k-\ell}}.$$
\end{corollary}

\begin{proof} It is easy to see that $\sum_{\ell=0}^k \overline{I^\ell}\cdot \overline{J^{k-\ell}} \subseteq \overline{(I+J)^k}$ (see, for example, \cite[Proposition 10.5.2]{SH2006}). We shall prove the other inclusion. Since $\overline{(I+J)^k}$ is a monomial ideal, it suffices to show that all monomials in $\overline{(I+J)^k}$ belong to $\sum\limits_{\ell = 0 }^k \overline{I^\ell} \cdot \overline{J^{k-\ell}}.$
	
Let $M_1$, $M_2$ and $M$ be exponent matrices as in Theorem \ref{thm.binRAT} (and using the same notations).
Consider any monomial $x^\a \in \overline{(I+J)^k}$, for $\a \in \ZZ_{\ge 0}^n$. By \cite[Proposition 1.1]{HT2019}, we have $\nu^*_\a(M) \ge k$. This condition states that the optimal solution to the linear programming problem (\ref{eq.LP}) is at least $k$.
	
Suppose that $\y = (y_1, \dots, y_m) \in \RR^m_{\ge 0}$ is a vector that gives the optimal solution to (\ref{eq.LP}). Then, we have
	$$M \cdot \y \le \a \text{ and } \1^m \cdot \y \ge k.$$
As before, write $\a = (\alpha, \0) + (\0, \beta)$, where $\alpha \in \ZZ_{\ge 0}^r$ and $\beta \in \ZZ_{\ge 0}^s$, and $\y = (\y_1, \0) + (\0, \y_2)$, where $\y_1 \in \RR^{m_1}_{\ge 0}$ and $\y_2 \in \RR^{m_2}_{\ge 0}$. As observed in the proof of Theorem \ref{thm.binRAT}, the optimization problem (\ref{eq.LP}) is equivalent to the two optimization problems ($\dagger$) and ($\sharp$).
	
Let $\ell$ be the optimal solution to ($\dagger$); that is, $\nu^*_\alpha(I) = \nu^*_\alpha(M_1) = \ell$. Then, the optimal solution to ($\sharp$) satisfies $\nu^*_\beta(M_2) \ge k-\ell$. By the hypotheses, $\ell \in \ZZ$ and $k-\ell \in \ZZ$. The inequalities $\nu^*_\alpha(M_1) \ge \ell$ and $\nu^*_\beta(M_2) \ge k-\ell$, by \cite[Proposition 1.1]{HT2019}, then imply that $x^{\alpha} \in \overline{I^\ell}$ and $x^\beta \in \overline{J^{k - \ell}}$. As a consequence, we get that $x^\a \in \overline{I^\ell} \cdot \overline{J^{k-\ell}}$, which belong to the right hand side of the desired equality. The assertion is proved.
\end{proof}

As immediate consequences of Corollary \ref{thm.BinomialExpansion}, we obtain the following corollaries which generalize \cite[Theorem 2.1]{MT2021}.

\begin{corollary} \label{cor.BIN}
	Let $I \subseteq A$ and $J \subseteq B$ be monomial ideals. Suppose that $I$ is squarefree and $I^{(k)} = \overline{I^k}$ for all $k \in \NN$. Then, for any $k \in \NN$, we have
	$$\overline{(I+J)^k} = \sum_{\ell=0}^k \overline{I^\ell}\cdot \overline{J^{k-\ell}}.$$
\end{corollary}

\begin{proof} By \cite[Proposition 1.1]{HT2019}, for any $\alpha \in \ZZ_{\ge 0}^r$, we have $\nu^*_\alpha(I) = \tau^*_\alpha(I) = \tau_\alpha(I) \in \ZZ$. The conclusion now follows from Corollary \ref{thm.BinomialExpansion}.
\end{proof}

\begin{corollary}[{\cite[Theorem 2.1]{MT2021}}]
	Let $I \subseteq A$ be a normally torsion-free squarefree monomial ideal, and let $J \subseteq B$ be an arbitrary monomial ideal. Then, for any $k \in \NN$, we have
	$$\overline{(I+J)^k} = \sum_{\ell=0}^k I^\ell \cdot \overline{J^{k-\ell}}.$$
\end{corollary}

\begin{proof} Since $I$ is a squarefree monomial ideal, we have $I^k \subseteq \overline{I^k} \subseteq I^{(k)}$ for all $k \in \NN$. The assumption that $I$ is normally torsion-free then implies that $I^k = \overline{I^k} = I^{(k)}$ for all $k \in \NN$. The assertion now follows from Corollary \ref{cor.BIN}.
\end{proof}

Making use of jumping numbers of powers of $I$ and $J$, as a consequence of Theorem \ref{thm.jump}, we present another sufficient condition for the binomial expansion (\ref{eq.binIC}) to hold. The authors again thank an anonymous referee for this observation.

\begin{corollary}
	\label{cor.bin_jump}
 Let $I \subseteq A$ and $J \subseteq B$ be monomial ideals, and let $k \in \NN$. Suppose that the jumping numbers of either $I$ or $J$ on $[0,k]$ are all integers. Then, we have
 $$\overline{(I+J)^k} = \sum_{\ell = 0}^k \overline{I^\ell} \cdot \overline{J^{k-\ell}}.$$
\end{corollary}

\begin{proof} As in the proof of Corollary \ref{thm.BinomialExpansion}, we have $\sum_{\ell = 0}^k \overline{I^\ell} \cdot \overline{J^{k-\ell}} \subseteq \overline{(I+J)^k}$. The other inclusion follows immediately from the hypothesis on jumping numbers of $I$ or $J$ and Theorem \ref{thm.jump}.
\end{proof}

\begin{corollary}
	\label{cor.k=1}
	Let $I \subseteq A$ and $J \subseteq B$ be monomial ideal. Then,
	$$\overline{I+J} = \overline{I} + \overline{J}$$
	if and only if either $I$ or $J$ does not have any jumping number in $(0,1)$.
\end{corollary}

\begin{proof}
	One implication follows directly from Corollary \ref{cor.bin_jump}. We shall establish the other implication; that is, if $\overline{I+J} = \overline{I} + \overline{J}$ then either $I$ or $J$ does not have any jumping number in $(0,1)$. Suppose, on the contrary, that both $I$ and $J$ have jumping numbers in $(0,1)$. 
	
	Let $r \in (0,1)$ be a jumping number of $I$. If there is a jumping number of $J$ lying in $[1-r,1)$, then it can be seen that $I_r \not= I_1 = \overline{I}$ and $J_{1-r} \not= J_1 = \overline{J}$. Thus, $I_rJ_{1-r}$ is not contained in $\overline{I} + \overline{J}$, which is a contradiction to Theorem \ref{thm.jump}. We have shown that, if $r \in (0,1)$ is a jumping number of $I$, then all jumping numbers in $(0,1)$ of $J$ are smaller than $1-r$. Hence, by interchanging $I$ and $J$ if necessary, we now may assume that $r < 1/2$. We may also take $r$ to be the smallest jumping number of $I$ and $J$ in $(0,1)$.
	
	It follows from \cite[Theorem 5.9.(3)]{D+2021} that $nr$ is a jumping number of $I$ for all $n \in \NN$. Let $mr$ be the largest multiple of $r$ that is strictly less than 1. Particularly, $1-mr \le r$. This implies that $J$ cannot have any jumping number lying in $(0,1-mr)$. Therefore, $J$ must have a jumping number lying in $[1-mr, 1)$. By the same argument as above, we conclude that $I_{mr}J_{1-mr}$ is not contained in $\overline{I} + \overline{J}$, which is a contradiction to Theorem \ref{thm.jump}. The assertion is proved.
\end{proof}

The condition that the jumping numbers of either $I$ or $J$ on $[0,k]$ are all integers in Corollary \ref{cor.bin_jump} is not a necessary condition, as illustrated in the following example. We thank an anonymous referee and Jonathan Monta\~no for providing us with this example.

\begin{example}
	\label{ex.notNess}
	Consider $I = (xy,yz,zx) \subseteq \kk[x,y,z]$ and $J = (ab,bc,ca) \subseteq \kk[a,b,c]$. Direct computation shows that
	\begin{align*} 
		(I+J)_2 & = \overline{(I+J)^2} \\
		& = (b^2c^2, abc^2, ab^2c, a^2c^2, a^2bc, a^2b^2, yzbc, yzac, yzab, y^2z^2, xzbc, xzac, xzab, \\
		& \quad \  \  \  xybc, xyac, xyab, xyz^2, xy^2z, x^2z^2, x^2yz, x^2y^2) \\
		& = \overline{I^2} + \overline{I} \cdot \overline{J} + \overline{J^2} = I_2 + I_1 J_1 + J_2.
	\end{align*}
	On the other hand, it can be seen that $\frac{3}{2}$ is a jumping number of both $I$ and $J$ on $[0,2]$.
\end{example}

%%%%%%%%%%%%%%%%%%%%%%%%%%%%%%%%%

\section{Depth and regularity} \label{sec.depthreg}

Depth and regularity are perhaps among the most important invariant associated to ideals and modules. In this section, we shall use the binomial expansions established in the previous section to give bounds and precise formulas for the depth and regularity of rational powers $(I+J)_u$ and the integral closures $\overline{(I+J)^k}$.

We start with a general bounds for the depth and regularity of a rational power $(I+J)_u$, whose proof is an easy adaptation of that of \cite[Theorem 4.2]{HNTT2020} together with the binomial expansion in Theorem \ref{thm.jump}.

\begin{theorem}
	\label{thm.boundsRegDepthJump}
	Let $u \in \QQ_+$, and assume that the jumping numbers of $I$ on the interval $[0,u]$ are $0 \le k_0 < k_1 < \dots < k_n \le u$. Then, we have
		\begin{enumerate}
		\item[{\rm (1)}] $\depth S/(I+J)_u \ge$ \newline
		\hspace*{3ex} ${\displaystyle \min_{\substack{\ell \in [1,n-1] \\ \theta \in [1,n]}} \{\depth A/I_{k_{n-\ell}} + \depth B/J_{u-k_{n-\ell}} + 1, \depth A/I_{k_{n-\theta+1}} + \depth B/J_{u-k_{n-\theta}}\}}$,
		\item[{\rm (2)}] $\reg S/(I+J)_u \le$ \newline
		\hspace*{3ex} ${\displaystyle \max_{\substack{\ell \in [1,n-1] \\ \theta \in [1,n]}} \{\reg A/I_{k_{n-\ell}} + \reg B/J_{u-k_{n-\ell}} + 1, \reg A/I_{k_{n-\theta+1}} + \reg B/J_{u-k_{n-\theta}}\}}$.
	\end{enumerate}
\end{theorem}

\begin{proof}
	By Theorem \ref{thm.jump}, we have
	$$(I+J)_u = I_{k_n} \cdot J_{u-k_n} + I_{k_{n-1}} \cdot J_{u-k_{n-1}} + \dots + I_{k_0} \cdot J_{u-k_0}.$$
	For $0 \le t \le n$, set
	$$P_{u,t} = I_{k_n}\cdot J_{u-k_n} + I_{k_{n-1}}\cdot J_{u-k_{n-1}} + \dots + I_{k_{n-t}} \cdot J_{u-k_{n-t}}.$$
	
Observe that
\begin{enumerate}
	\item[(a)] $P_{u,t} = P_{u,t-1} + I_{k_{n-t}} \cdot J_{u-k_{n-t}}$ for $1 \le t \le n$; and
	\item[(b)] $P_{u,t-1} \cap I_{k_{n-t}} \cdot J_{j-k_{n-t}} = I_{k_{n-t+1}} \cdot J_{u-k_{n-t}}.$
\end{enumerate}
\noindent Indeed, (a) is obvious. To see (b), notice first that, by \cite[Lemma 5.1]{D+2021}, $P_{u,t-1} \subseteq I_{k_{n-t+1}}$. Thus,
$$P_{u,t-1} \cap I_{k_{n-t}} \cdot J_{j-k_{n-t}} \subseteq I_{k_{n-t+1}} \cap J_{u-k_{n-t}} = I_{k_{n-t+1}} \cdot J_{u-k_{n-t}}.$$
On the other hand, also by \cite[Lemma 5.1]{D+2021}, we have
$$I_{k_{n-t+1}} \cdot J_{u-k_{n-t}} \subseteq I_{k_{n-t+1}} \cdot J_{u-k_{n-t+1}} \subseteq P_{u,t-1} \text{ and } I_{k_{n-t+1}} \cdot J_{u-k_{n-t}} \subseteq I_{k_{n-t}} \cdot J_{u-k_{n-t}}.$$
Therefore,
$$P_{u,t-1} \cap I_{k_{n-t}} \cdot J_{j-k_{n-t}} = I_{k_{n-t+1}} \cdot J_{u-k_{n-t}}.$$
The desired inequality for depth and regularity now follow by tracing through the exact sequences
$$0 \rightarrow S/I_{k_{n-t+1}} \cdot J_{u-k_{n-t}} \rightarrow S/P_{u,t-1} \oplus S/I_{k_{n-t}} \cdot J_{u-k_{n-t}} \rightarrow S/P_{u,t} \rightarrow 0,$$
and making use of \cite[Lemmas 2.2 and 3.2]{HoaT2010}.
\end{proof}

By a similar line of arguments as in Theorem \ref{thm.boundsRegDepthJump}, Corollaries \ref{thm.BinomialExpansion} and \ref{cor.bin_jump} then give the following consequence.

\begin{lemma}[\protect{See \cite[Theorem 4.2]{HNTT2020}}]
	\label{lem.boundsRegDepth}
	Let $k \in \NN$. Suppose that at least one of the following conditions holds:
	\begin{enumerate}
		\item for every nonnegative integral vector $\alpha \in \ZZ_{\ge 0}^r$, $\nu^*_\alpha(I) \in \ZZ$; or
		\item the jumping numbers on $[0,k]$ of either $I$ or $J$ are all integers.
	\end{enumerate}
	Then, 
	\begin{enumerate}
		\item[{\rm (1)}] $\depth S/\overline{(I+J)^k} \ge$ \newline
		\hspace*{3ex} ${\displaystyle \min_{\substack{i \in [1,k-1] \\ j \in [1,k]}} \{\depth A/\overline{I^{k-i}} + \depth B/\overline{J^i} + 1, \depth A/\overline{I^{k-j+1}} + \depth B/\overline{J^j}\}}$,
		\item[{\rm (2)}] $\reg S/\overline{(I+J)^k} \le$ \newline
		\hspace*{3ex} ${\displaystyle \max_{\substack{i \in [1,k-1] \\ j \in [1,k]}} \{\reg A/\overline{I^{k-i}} + \reg B/\overline{J^i} + 1, \reg A/\overline{I^{k-j+1}} + \reg B/\overline{J^j}\}}$.
	\end{enumerate}
\end{lemma}

We will show that the inequalities in Lemma \ref{lem.boundsRegDepth} are in fact equalities. To this end, for $0 \le t \le k$, as before, set
$$P_{k,t} = \overline{I^k}\cdot \overline{J^0} + \overline{I^{k-1}}\cdot \overline{J} + \dots + \overline{I^{k-t}}\cdot \overline{J^t}.$$
Then, as in the proof of Theorem \ref{thm.boundsRegDepthJump} (see also \cite[Theorem 4.2]{HNTT2020}), we have
\begin{enumerate}
	\item[(a)] $P_{k,t} = P_{k,t-1} + \overline{I^{k-t}}\cdot \overline{J^t}$ for $1 \le t \le k$;
	\item[(b)] $P_{k,t-1} \cap \overline{I^{k-t}}\cdot \overline{J^t} = \overline{I^{k-t+1}}\cdot \overline{J^t}.$
\end{enumerate}
These decomposition allow us to evoke the following notion and property of a Betti splitting to investigate the depth and regularity of $P_{k,t}$; see \cite{FHVT2009}. For homogeneous ideals $P, I, J$ in $S$ such that $P = I+J$, the sum $P = I+J$ is called a \emph{Betti splitting} if the graded Betti numbers of $P, I, J$ and $I \cap J$ satisfy the following relation:
$$\beta_{i,j}(P) = \beta_{i,j}(I) + \beta_{i,j}(J) + \beta_{i-1,j}(I \cap J) \text{ for all } i \ge 0 \text{ and } j \in \ZZ.$$

\begin{lemma}[{\cite[Corollary 2.2]{FHVT2009}}]
	\label{lem.BSplit}
	Let $P = I+J$ be a Betti splitting in $S$. Then,
	\begin{enumerate}
		\item[{\rm (a)}] $\depth S/P = \min\{\depth S/I, \depth S/J, \depth S/I \cap J - 1\},$
		\item[{\rm (b)}] $\reg S/P = \max\{\reg S/I, \reg S/J, \reg S/I \cap J - 1\}.$
	\end{enumerate}
\end{lemma}

As in the proof of \cite[Theorem 5.3]{HNTT2020}, to establish the equality in Lemma \ref{lem.boundsRegDepth}, it suffices to show that $P_{k,t} = P_{k,t-1} + \overline{I^{k-t}}\cdot \overline{J^t}$ is a Betti splitting. This is also how we proceed. It turns out that Betti splitting can be characterized by \emph{Tor-vanishing} homomorphisms. We shall recall the following necessary terminology and results from \cite{FHVT2009, HNTT2020, NV2019}.

\begin{definition} [{See \cite{HNTT2020, NV2019}}] \quad
	\label{def.Torvanishing}
	\begin{enumerate}
		\item We say that a homomorphism $\phi: M \rightarrow N$ of graded $S$-modules is \emph{Tor-vanishing} if
		$$\Tor^S_i(\kk, \phi) = 0 \text{ for all } i \ge 0.$$
		\item We say that a filtration $\{Q_k\}_{k \in \NN}$ of $S$-modules is a \emph{Tor-vanishing filtration} if, for all $k \ge 1$, the inclusion map $Q_k \rightarrow Q_{k-1}$ is \emph{Tor-vanishing}. That is,
		$\Tor^S_i(\kk, Q_k) \rightarrow \Tor^S_i(\kk, Q_{k-1})$ is the zero map for all $k \ge 1$.
	\end{enumerate}
\end{definition}

\begin{lemma}[{\cite[Proposition 2.1]{FHVT2009}}]
	\label{lem.BSvsTV}
	The following conditions are equivalent:
	\begin{enumerate}
		\item[{\rm (1)}] The decomposition $P = I+J$ is a Betti splitting, and
		\item[{\rm (2)}] The inclusion maps $I \cap J \rightarrow I$ and $I \cap J \rightarrow J$ are Tor-vanishing.
	\end{enumerate}
\end{lemma}

In light of Lemmas \ref{lem.BSplit} and \ref{lem.BSvsTV}, to show that the inequality in Lemma \ref{lem.boundsRegDepth} are equality, our argument is based on the following essential fact: the family of integral closures of powers of a monomial ideal is a Tor-vanishing filtration.

\begin{lemma}
	\label{lem.Torvanishing}
	Let $I$ be a monomial ideal in $A$.
	\begin{enumerate}
		\item[{\rm (1)}] $\{\overline{I^k}\}_{k \in \NN}$ is a Tor-vanishing filtration of ideals.
		\item[{\rm (2)}] Assuming that $\mm$ is the maximal homogeneous ideal in $A$, for any $k \in \NN$, we have
		$$\overline{I^{k}} \subseteq \mm \cdot \overline{I^{k-1}}.$$
	\end{enumerate}
\end{lemma}

\begin{proof} (1) Recall from \cite{HNTT2020} that for a monomial ideal $I$, $\delta^*(I)$ denotes the ideal generated by elements of the form $f/x$, where $f$ is a minimal monomial generator of $I$ and $x$ is a variable dividing $f$. By \cite[Proposition 4.4 and Lemma 4.2]{NV2019} (see also \cite[Proposition 3.5]{AM2019}), it suffices to show that
	\begin{align}
		\delta^*(\overline{I^k}) \subseteq \overline{I^{k-1}}. \label{eq.delta}
	\end{align}
	
	Let $M$ be any minimal monomial generator of $\overline{I^k}$. Then,
	$$M^r \in (I^k)^r = I^{kr} \text{ for some } r \in \NN.$$
	That is, there exist monomials $f_1, \dots, f_{kr} \in I$ (not necessarily distinct) such that $M^r = f_1 \dots f_{kr}$.
	
	Let $x$ be a variable dividing $M$. Clearly, $x^r ~\big|~ M^r$. This implies that there exist $i_1, \dots, i_s$, for some $s \le r$ such that $x$ appears in $f_{i_1} \dots f_{i_s}$ with powers at least $r$. By considering the product of $f_i$'s for $i \not= i_1, \dots, i_s$, it is easy to see that $(M/x)^r \in I^{kr-r} = (I^{k-1})^r$. Thus, $M/x \in \overline{I^{k-1}}$. We have established (\ref{eq.delta}).
	
	(2) The statement is trivial for $k = 1$. Suppose that $k \ge 2$. By definition, we have
	$$\overline{I^k} \subseteq \mm \cdot \delta^*(\overline{I^k}).$$
	The desired containment now follows from \eqref{eq.delta}.
\end{proof}

As a consequence of Lemma \ref{lem.Torvanishing}, we immediately obtain the following containment, which is of independent interest.

\begin{corollary} \label{cor.containment}
	For any positive integer $e$ we have
	$$\overline{I^{k+e}} \subseteq \mm^e \cdot \overline{I^{k}}.$$
\end{corollary}

\begin{proof}
The assertion follows from a repeated application of part (2) of Lemma \ref{lem.Torvanishing}.
\end{proof}

We are now ready to state our last main result.

\begin{theorem}
	\label{thm.DepthReg}
	Let $I \subseteq A$ and $J \subseteq B$ be monomial ideals, and let $k \in \NN$. Suppose that at least one of the following conditions holds:
	\begin{enumerate}
		\item for every nonnegative integral vector $\alpha \in \ZZ_{\ge 0}^r$, $\nu^*_\alpha(I) \in \ZZ$; or
		\item the jumping numbers on $[0,k]$ of either $I$ or $J$ are all integers.
	\end{enumerate}
	Then, we have
		\begin{enumerate}
		\item[{\rm (1)}] $\depth S/\overline{(I+J)^k}=$ \newline
		\hspace*{3ex} ${\displaystyle \min_{\substack{i \in [1,k-1] \\ j \in [1,k]}} \{\depth A/\overline{I^{k-i}} + \depth B/\overline{J^i} + 1, \depth A/\overline{I^{k-j+1}} + \depth B/\overline{J^j}\}}$,
		\item[{\rm (2)}] $\reg S/\overline{(I+J)^k} =$ \newline
		\hspace*{3ex} ${\displaystyle \max_{\substack{i \in [1,k-1] \\ j \in [1,k]}} \{\reg A/\overline{I^{k-i}} + \reg B/\overline{J^i} + 1, \reg A/\overline{I^{k-j+1}} + \reg B/\overline{J^j}\}}$.
	\end{enumerate}
\end{theorem}

\begin{proof} By Lemma \ref{lem.Torvanishing}, we have that $\{\overline{I^k}\}_{k \in \NN}$ and $\{\overline{J^k}\}_{k \in \NN}$ are Tor-vanishing filtrations of ideals in $A$ and $B$, respectively. It then follows from the proof of \cite[Theorem 5.3]{HNTT2020} that $P_{k,t} = P_{k,t-1} + \overline{I^{k-t}} \cdot \overline{J^t}$ is a Betti splitting for all $1 \le t \le k$. Noticing again that $P_{k,t-1} \cap \overline{I^{k-t}} \cdot \overline{J^t} = \overline{I^{k-t+1}} \cdot \overline{J^t}$. Now, applying Lemma \ref{lem.BSplit} and \cite[Lemmas 2.2 and 2.3]{HoaT2010} in the same way as in the proof of \cite[Theorems 4.2 and 5.3]{HNTT2020}, we obtain the desired equality.
\end{proof}

As noted before, the condition that $\nu^*_\alpha(I) \in \ZZ$ in Theorem \ref{thm.DepthReg} is satisfied, for instance, when $I^{(k)} = \overline{I^k}$ for all $k \in \NN$.

\begin{corollary} \label{cor.DepthReg}
	Let $I \subseteq A$ and $J \subseteq B$ be monomial ideals. Suppose that $I^{(k)} = \overline{I^k}$ for all $k \in \NN$. Then, the inequalities in Lemma \ref{lem.boundsRegDepth} are equality.
\end{corollary}

\begin{remark} It would be of interest to know when the inequality in Theorem \ref{thm.boundsRegDepthJump} are equality. Since rational powers at jumping numbers of $I$ and $J$ do not necessarily give filtration of ideals, it is not clear how techniques of Betti splitting and Tor-vanishing would generalize to this case.
\end{remark}

%%%%%%%%%%%%%%%%%%%%%%%%%%%%%%
%%%%%%%%%%%%%%%%%%%%%%%%%%%%%%

\vspace{1em}

\begin{center}
	\textbf{Declarations}
\end{center}

\noindent\textbf{Ethical Approval:} Not applicable.

\noindent\textbf{Competing interests:} There are no competing interests, of either financial or personal nature.

\noindent\textbf{Authors' contributions:} The authors, A.B. and T.H.H., contribute equally in the investigation, methodology, formal analysis and writing of this article.

\noindent\textbf{Funding:} A.B. is partially supported by DST INSPIRE Faculty Fellowship and CPDA of IIT Kharagpur. T.H. acknowledges supports from Louisiana Board of Regents and the Simons Foundation.

\noindent\textbf{Availability of data and materials:} Data sharing not applicable to this article as no data-sets were generated or analyzed during the current study.

%%%%%%%%%%%%%%%%%%%%%%%%%%%%%%%%%

%\section{Containment of ideals} \label{sec.contain}
%$$\overline{I^k} \subseteq \mm \cdot \delta^*(\overline{I^k}).$$

%%%%%%%%%%%%%%%%%%%%%%%%%%%%%%%%

\end{document}